\documentclass[12pt]{amsart}

\usepackage[english]{babel}
\usepackage[utf8x]{inputenc}
\usepackage{amsmath, fullpage}
\usepackage{cite}
\usepackage{tikz-cd}
\usepackage{lipsum}
\usepackage{url}

\theoremstyle{plain}
\newtheorem{thm}{Theorem}[section]

\newtheorem{prop}[thm]{Proposition}

\newtheorem{def-thm}[thm]{Definition-Theorem}
\newtheorem{def-lem}[thm]{Definition-Lemma}

\newtheoremstyle{romanstyle}          
{3pt}{3pt}                          
{\normalfont}                       
{}                                  
{\normalfont}                       
{.}                                 
{0.5em}                             
{}                                  

\theoremstyle{romanstyle}
\newtheorem{definition}[thm]{Definition}

\newtheorem{rk}[thm]{Remark}

\newtheorem*{acknowledgement}{Acknowledgement}

\newcommand{\Spec}{\operatorname{Spec}}

\newcommand{\Spf}
{\operatorname{Spf}}
\newcommand{\Res}{\operatorname{Res}}
\newcommand{\im}{\operatorname{Im}}

\newcommand{\Aut}{\operatorname{Aut}}

\newcommand{\Der}{\operatorname{Der}}
\newcommand{\Ker}{\operatorname{Ker}}
\newcommand{\End}{\operatorname{End}}
\newcommand{\Hom}{\operatorname{Hom}}

\newcommand{\Span}{\operatorname{Span}}

\newcommand{\Ind}{\operatorname{Ind}}
\newcommand{\id}{\operatorname{id}}

\newcommand{\gp}{\operatorname{gp}}

\newcommand{\0}{|0\rangle}

\newcommand{\vp}{\varphi}

\newcommand{\ul}{\underline}

\newcommand{\et}{\textrm{\'{e}t}}

\newcommand{\cD}{\mathcal{D}}

\newcommand{\cI}{\mathcal{I}}

\newcommand{\cO}{\mathcal{O}}

\newcommand{\bA}{\mathbb{A}}

\newcommand{\bC}{\mathbb{C}}

\newcommand{\bN}{\mathbb{N}}

\newcommand{\bZ}{\mathbb{Z}}

\newcommand{\ch}{\operatorname{ch}}
\newcommand{\sA}{\mathsf{A}}

\makeatletter
\newcommand{\AtEndAffiliation}{%
	\def\enddoc@text{\ifx\@empty\@addresses\else
		\bigskip\bigskip
		\par\noindent
		\@author \par 
		\@setaddresses 
		\fi}%
}
\makeatother

\title{(Log) Chiral de Rham complex and $A_n$-singular surfaces}
\author{Xi-Chuan Tan}
\address{Department of Mathematics, University of Tsukuba, Ibaraki 305-8571, Japan}
\AtEndAffiliation

\begin{document}
	
	
	\begin{abstract}
		We present a construction of the chiral de Rham complex over an algebraic surface with at most rational singularities of $A_n$-type.
		An explicit formula for the character of the chiral structure sheaf is also provided.
	\end{abstract}
	
	\maketitle
	\thispagestyle{empty}
	
	\section{Introduction}
	Originating from string theory and two-dimensional conformal field theory  (\cite{friedan1986conformal}),
	the $\beta\gamma-bc$ system is a vertex algebra generated by two even fields $\beta(z)$, $\gamma(z)$ and two odd fields $b(z)$, $c(z)$,
	with the following operator product expansions (OPE):
	$\beta(z)\gamma(w)\sim1/(z-w)$,
	$b(z)c(w)\sim1/(z-w)$.
	For a regular complex variety (or complex manifold) $X$ of dimension $n$,
	the chiral de Rham complex $\Omega^{\ch}_X$ over $X$ is defined as the sheaf of vertex algebras associated to the $n$-fold tensor product of the $\beta\gamma-bc$ system (\cite{malikov1999chiral} \cite{song2016global}),
	where $\gamma$'s play the role of local coordinates on $X$.
	
	The chiral de Rham complex $\Omega^{\ch}_X$ is equipped with a chiral differential $d_0$.
	A remarkable result, proved in \cite{malikov1999chiral},
	is that $(\Omega^{\ch}_X,d_0)$ is quasi-isomorphic to the usual de Rham complex on $X$.
	Rather than the usual structure sheaf $\cO_X$,
	it will be helpful to consider the chiral structure sheaf $\cO^{\ch}_X$,
	which is associated to the $n$-fold tensor product of the $\beta\gamma$ system.
	In this setting,
	$\Omega^{\ch}_X$ is an $\cO^{\ch}_X$-module.
	Indeed, this will be more comprehensible when viewed through the superscheme $\Pi TX$,
	where $TX$
	is the tangent bundle over $X$ and $\Pi$ denotes the parity change functor.
	Over $\Pi TX$,
	locally we have $n$ even coordinates $\gamma^1,...,\gamma^n$ 
	and odd coordinates
	$c^1=d\gamma^1,...,c^n=d\gamma^n$.
	The chiral de Rham complex is locally determined by these coordinates $\gamma^i, c^i$ and their respective duals $\beta^i, b^i$ ($i=1,...,n$).
	
	Rational singularities on a surface are classified into types $A_n$, $D_n$, $E_6$, $E_7$ and $E_8$  (\cite{reid2012val}),
	determined by the intersections of exceptional curves in the blowing-up.
	Using logarithmic geometry,
	we construct a coordinate system near a singularity of type $A_n$ on a surface.
	A formally \'{e}tale map from the infinitesimal 2-dimensional disc $D^2=\Spec \bC[[x,y]]$ to a surface $X$ is defined to be a coordinate at the image of the origin of the disc.
	And continuous automorphisms of $D^2$ are regarded as the local coordinate changes on $X$.
	By identifying $\gamma$'s and $c$'s with the coordinates on the superscheme $\Pi TX$,
	we can study the transformations of the fields $\beta(z)$, $\gamma(z)$, $b(z)$ and $c(z)$ under coordinate changes.
	We will show that these data are compatible with the classical setting and together form a sheaf of vertex algebras, which we denote by $\Omega^{\ch}_X$ as well.
	Logarithmic geometry applies to $A_n$-singularities because these singular surfaces are toric, which fit naturally into the framework of log schemes.
	Specifically, there exists a monoid $Q$ with nice properties, 
	such that $X$ is locally of the form
	$\Spec\bC[Q]$
	near the singularity,
	endowed with the log structure $Q\to\bC[Q]$.
	Logarithmic geometry provides a tool to study the above coordinates on $\Pi TX$, 
	and thus helps to construct the chiral de Rham complex on singular surfaces of type $A_n$.
	
	At the end of this paper, we provide a recursive formula for the dimensions of homogeneous components of the (log) chiral structure sheaf.
	
	The base field of all analytic  and algebraic objects is the complex field $\bC$, unless specified.

	\begin{acknowledgement}
		The author thanks Scott Carnahan for many helpful comments.
		This work was supported by JST, the establishment of university fellowships towards the creation of science technology innovation, Grant Number JPMJFS2106.
	\end{acknowledgement}

	\section{Preliminaries}
	
	\subsection{Vertex algebras}\label{va}
	
	We recall the necessary notions of vertex algebras needed in this paper.
	For further topics and details, see \cite{frenkelvertex}.
	
	\begin{definition}
		For a complex vector space $V$,
		a formal power series
		$$A(z)=\sum_{j\in\bZ}A_jz^{-j}\in(\End V)[[z^{\pm1}]]$$
		is called a \textit{field} on $V$ if
		for any $v\in V$, we have $A_j\cdot v=0$ for $j\gg0$.
	\end{definition}
	
	\begin{definition}
		A \textit{vertex algebra} is a vector space $V$ over a field of characteristic $0$, equipped with the \textit{vacuum vector} $\0$, 
		the \textit{translation operator} $T\in\End V$, and the \textit{vertex operation} $Y(-,z):V\to(\End V)[[z^{\pm1}]]$ which  linearly takes each $A\in V$ to a field $Y(A,z)=\sum_{n\in\bZ}A_{(n)}z^{-n-1}$, such that $(V,\0,T,Y)$ satisfies the following axioms:
		\\
		1) Vacuum axiom: $Y(\0,z)=\id_V$. Furthermore, $A_{(n)}\0=0$ for $n\geq0$ and $A_{(-1)}\0=A$;
		\\
		2) Translation axiom: for any $A\in V$, $Y(TA,z)=\partial_z Y(A,z)$;
		\\
		3) Locality axiom: for any $A,B\in V$, there exists some $N\in\bZ_{\geq0}$ such that $$(z-w)^N[Y(A,z),Y(B,w)]=0.$$
	\end{definition}
	
	For convenience of notation, we denote a vertex algebra $(V,\0,T,Y)$ simply by $V$.
	And the field $Y(A,z)$ for $A\in V$ is sometimes denoted by $A(z)$.
	A vertex algebra $V$ is said to be $\bZ$-\textit{graded} if $V$ is a $\bZ$-graded vector space, $\0$ is a vector of degree $0$, $T$ is a linear operator of degree $1$, and for $A\in V_m$,
	we have $\deg A_{(n)}=-n-1+m$. 
	
	\begin{definition}
		A $\bZ$-graded vertex algebra $V$ is called \textit{conformal of central charge} $c\in\bC$, if there is a given non-zero \textit{conformal vector} $\omega\in V_2$ such that the Fourier coefficients $L_n$ of the corresponding vertex operator $Y(\omega,z)=\sum_{n\in\bZ}L_nz^{-n-2}$ satisfy the following conditions:
		\\
		1) $[L_n,L_m]=(n-m)L_{n+m}+\frac{n^3-n}{12}\delta_{n,-m}c\id_V$;
		\\
		2) $L_{-1}=T$, $L_0|_{V_n}=n\id_{V_n}$.
		
	\end{definition}
	
	The following notion of vertex subalgebra is from \cite{lepowsky2012introduction}-3.9.1.
	
	\begin{definition}
		A \textit{vertex subalgebra} of a vertex algebra $(V,\0,T,Y)$ is a vector subspace $U$ of $V$
		such that $(U,\0,T|_U,Y|_U)$ is itself a vertex algebra.
	\end{definition}
	
	We attach a Lie algebra to each given vertex algebra,
	which will play an important role in the construction of chiral structure sheaf in sequel.
	More details can be found in \cite{frenkelvertex}-4.1.
	
	Let $V$ be a vertex algebra.
	Set
	$$U'(V)=(V\otimes\bC[t^{\pm1}])/\im\partial$$
	where $\partial=T\otimes\id+\id\otimes\partial_t$ is a linear operator on $V\otimes\bC[t^{\pm1}]$.
	Denote by the projection of $A\otimes t^n\in V\otimes\bC[t^{\pm1}]$ in $U'(V)$ by $A_{[n]}$.
	We endow the relation $(TA)_{[n]}=-nA_{[n]}$ to $U'(V)$,
	and then there is a linear map
	$$U'(V)\to\End V,\quad A_{[n]}\mapsto A_{(n)}.$$
	Define a bilinear map
	$[\ ,\ ]:U'(V)\otimes U'(V)\to U'(V)$ by
	\begin{equation}\label{brack}
		A_{[m]}\otimes B_{[k]}\mapsto
		[A_{[m]},B_{[k]}]=\sum_{n\geq0}
		\left(
		\begin{matrix}
			m\\
			n
		\end{matrix}
		\right)
		(A_{(n)}B)_{[m+k-n]}.
	\end{equation}
	If $V$ is $\bZ$-graded, then $U'(V)$ is also $\bZ$-graded, by setting
	$\deg A_{[n]}=-n+\deg A-1$ for homogeneous $A\in V$.
	The linear map $U'(V)\to\End V$ preserves this gradation.
	
	The natural topology on $\bC[t^{\pm1}]$ is induced by taking $t^n\bC[t^{\pm1}]$ as basis of neighborhoods near $0$.
	Then the completion of $U'(V)$ with respect to the natural topology is
	$$U(V)=(V\otimes\bC((t)))/\im\partial.$$
	We have a linear map $U(V)\to\End V$,
	$$
	\sum_{n>N}f_nA_{[n]}\mapsto\Res_{z=0} Y(A,z)f(z)dz
	$$
	where $f(z)=\sum_{n>N} f_nz^n\in\bC((z))$, which
	extends $U'(V)\to\End V$.
	
	\begin{thm}[\cite{frenkelvertex}-4.1.2]
		The bilinear map (\ref{brack}) defines Lie algebra structures on $U'(V)$ and $U(V)$.
		Furthermore, the natural maps
		$U'(V)\to\End V$ and
		$U(V)\to\End V$ are Lie algebra homomorphisms.
	\end{thm}

	\subsection{Logarithmic structure}
	In this section we recall some basic notions in logarithmic geometry.
	The theory of log (short for logarithmic) schemes was originally developed in 
	\cite{kato}, while our main reference will be \cite{ogus2018lectures}.
	Throughout this paper, all monoids are assumed to be commutative.
	
	Let $X$ be a scheme and $M_X$ a sheaf of monoids on it.
	A \textit{log structure} on $X$ is a morphism of sheaves of monoids $\alpha:M_X\to\cO_X$ such that the restriction of $\alpha$ to $\alpha^{-1}(\cO^*_X)$ is an isomorphism.
	Note that the monoid structure on the coordinate ring is given by the multiplication. 
	Throughout the paper, a log structure on a scheme $X$ means a log structure on the small \'{e}tale site $X_\et$ unless specified. 
	
	\begin{definition}
		A \textit{log scheme} is a pair $(X,M_X)$,
		consisting of a scheme $X$ together with a log structure on it.
	\end{definition}
	
	If no confusion arises,
	we simply write $X$ as a log scheme,
	and we denote by $\ul X$ the underlying scheme of $X$.
	A morphism of sheaves of monoids $\alpha:M_X\to\cO_X$ is called a \textit{prelog structure} on $X$, 
	and we say that $(X,M_X)$ is a \textit{prelog scheme}.
	For a prelog structure $\alpha:M_X\to\cO_X$ on $X$,
	we denote by its associated log structure $\alpha^{\log} :M^{\log}_X \to\cO_X$,
	where $M^{\log}_X$ is the push-out of
	$$
	\begin{tikzcd}
		\mathcal{O}_X^* & \alpha^{-1}(\mathcal{O}_X^*) \arrow[r, "\alpha|_{\alpha^{-1}\mathcal{O}_X^*}"] \arrow[l, '] & M_X
	\end{tikzcd}
	$$
	in the category of sheaves of monoids on $X_\et$, endowed with 
	$$M^{\log}_X \to\cO_X,\quad (a,b)\mapsto \alpha(a)b.$$
	In particular if $M_X$ is the zero constant sheaf as a prelog structure,
	then the corresponding log structure is said to be \textit{trivial}.
	And indeed the trivial log structure is the inclusion $\cO^*_X\hookrightarrow\cO_X$.
	
	A \textit{log ring} is a morphism of monoids $\beta:P\to A$
	where $P$ is a monoid and $A$ is a ring viewed as a multiplicative monoid.
	It is sometimes denoted by $(A,P)$ in the present paper if no confusion arises. 
	If $P\to A$ is a log ring, then $\Spec(P\to A)$ is defined to be the log scheme whose underlying scheme is $\ul X:=\Spec A$, together with the log structure $P^{\log}\to \cO_X$ induced by $P\to A$ by viewing $P$ as a constant sheaf on $X_\et$. 
	For a \textit{morphism} $f:(A,P)\to (B,Q)$ \textit{of log rings}, 
	we mean a commutative diagram
	$$
	\begin{tikzcd}
		A \arrow[r, "f^\sharp"]          & B           \\
		P \arrow[u] \arrow[r, "f^\flat"] & Q. \arrow[u]
	\end{tikzcd}
	$$
	The above morphism $f$ is called an \textit{isomorphism} if $f^\sharp$ is an isomorphism of rings and $f^\flat$ is strict (i.e. the morphism 
	$P/P^*\to Q/Q^*$ induced by $f^\flat$ is an isomorphism).
	
	Let $f:X\to Y$ be a morphism of prelog schemes and $E$ an $\cO_X$-module. An $E$-\textit{valued derivation} of $X/Y$ is a pair $(D,\delta)$ where $D:\cO_X\to E$ is a morphism of abelian sheaves and $\delta:M_X\to E$ is a morphism of sheaves of monoids such that the following conditions are satisfied:
	\\
	1) $D(\alpha_X(m))=\alpha_X(m)\delta(m)$ for all local sections $m$ of $M_X$;
	\\
	2) $\delta(f^\flat(n))=0$ for all local sections $n$ of $f^{-1}M_Y$;
	\\
	3) $D(ab)=aD(b)+bD(a)$ for all local sections $a,b$ of $\cO_X$;
	\\
	4) $D(f^\sharp(c))=0$ for all local sections $c$ of $f^{-1}\cO_Y$.
	\\ \hspace*{\fill} \\
	Denote by $\Der_{X/Y}(E)$ the set of all such derivations.
	An analog of differentials of usual schemes is given as the following result.
	
	\begin{thm}[\cite{ogus2018lectures}-IV-1.2.4]
		Let $f:X\to Y$ be a morphism of prelog schemes. Then the functor $E\mapsto\Der_{X/Y}(E)$ is representable by an $\cO_X$-module $\Omega_{X/Y}^1$ endowed with a universal derivation $d\in\Der_{X/Y}(\Omega_{X/Y}^1)$.
	\end{thm}
	
	For a morphism of log schemes $X\to Y$,  we denote by $\Omega_{\underline{X}/\underline{Y}}^1$ the sheaf of relative differentials of the underlying schemes. 
	And for a monoid $Q$, let $Q^{\gp}$ be the group associated to $Q$,
	which is identified with the cokernel of the diagonal embedding $Q\to Q\oplus Q$.
	Two constructions of $\Omega_{X/Y}^1$ are provided in \cite{kato}-1.7:
	\\ 
	1) $\Omega_{X/Y}^1\simeq(\Omega_{\underline{X}/\underline{Y}}^1\oplus(\cO_X\otimes M_X^{gp}))/R$ where $R$ is the $\cO_X$-submodule generated by sections of the form
	$$(d\alpha_X(m),-\alpha_X(m)\otimes m),\quad \forall m\in M_X,$$
	$$(0,1\otimes f^\flat(n)),\quad \forall n\in f^{-1}M_Y.$$
	2) $\Omega_{X/Y}^1\simeq(\cO_X\otimes M_X^{gp})/(R_1+R_2)$ by $dm\mapsto 1\otimes m$ for $m\in M_X$. Here 
	\\
	$\bullet$ $R_1\subset \cO_X\otimes M_X^{gp}$ is the subsheaf of sections locally of the form
	$$\sum_i\alpha_X(m_i)\otimes m_i-\sum_i\alpha_X(m_i')\otimes m_i',$$
	for $\sum_i\alpha_X(m_i)=\sum_i\alpha_X(m_i')$. 
	\\
	$\bullet$ $R_2$ is the image of $\cO_X\otimes f^{-1}M_Y^{gp}\to\cO_X\otimes M_X^{gp}$.\\
	
	We explain above discussion briefly in the language of log rings,
	which will be used in later construction,
	referring to \cite{ogus2018lectures}-IV-1.1.2.
	Let $\theta:(\alpha:P\to A)\to (\beta:Q\to B)$ be a morphism of log rings.
	Then $\theta$ gives rise to a morphism of affine log schemes
	$X\to Y$
	where $X=\Spec(Q\to B)$, $Y=\Spec(P\to A)$.
	The module of global sections of the $\cO_X$-module $\Omega_{X/Y}^1$ is isomorphic to
	$$(\Omega_{B/A}^1\oplus(B\otimes Q^{\gp}/P^{\gp}))/R.$$
	Here $R$ is the submodule of $(\Omega_{B/A}^1\oplus(B\otimes Q^{\gp}/P^{\gp}))$
	generated by elements of the form
	$$(d\beta(q),-\beta(q)\otimes\bar{\pi}(q))\quad\text{for}\ q\in Q$$
	where
	$\bar{\pi}:Q\to Q^{\gp}/P^{\gp}$
	is the canonical map induced by $\pi:Q\to Q^{\gp}$.
	Or alternatively, it is isomorphic to 
	$$(B\otimes Q^{\gp}/P^{\gp})/R'$$
	where $R'$ is the submodule generated by elements of the form
	$$\sum_i \beta(q_i)\otimes q_i-\sum_i \beta(q_i')\otimes q_i'\quad\text{for} \ \sum_i \beta(q_i)=\sum_i \beta(q_i').$$

	A morphism $f:X\to Y$ of log schemes is \textit{strict} if the induced morphism of sheaves of monoids $(f^*M_Y)^{\log}\to M_X$ is an isomorphism.
	A \textit{log thickening} is a strict closed immersion $i:T'\to T$ of log schemes such that the square of the ideal sheaf $\cI$ of $T'$ in $T$ vanishes, and the subgroup $1+\cI$ of $\cO_{T}^*\simeq M_{T}^*$ operates freely on $M_T$. 
	A \textit{log thickening over} $f$ is a commutative diagram
	$$
	\begin{tikzcd}\label{thi}
		T' \arrow[d, "g"'] \arrow[r, "i"] & T \arrow[d, "h"] \\
		X \arrow[r, "f"]                  & Y               
	\end{tikzcd}
	$$
	where $i$ is a log thickening. 
	A \textit{deformation} of $g$ to $T$ is an element of 
	$$\operatorname{Def}_f(g,T):=\{\tilde{g}:T\to X:\tilde{g}\circ i=g,f\circ\tilde{g}=h\}.$$
	The morphism $f:X\to Y$ is \textit{formally smooth} (resp. \textit{unramified}, resp. \textit{\'{e}tale}) if for every log thickening $T'\to T$ over $f$, locally on $T$ there exists at least one (resp. at most one, resp. exactly one) deformation $\tilde{g}$ of $g$ to $T$.
	A morphism $f$ is \textit{smooth} (resp. \textit{\'{e}tale}) if it is formally
	smooth (resp. \'{e}tale) and satisfies the following conditions:
	\\
	1) $M_X$ and $M_Y$ are coherent;
	\\
	2) $\ul f$ is locally of finite presentation.
	\\ \hspace*{\fill} \\
	We list two useful criteria for smoothness and \'{e}taleness from \cite{ogus2018lectures}-IV-3.1.9\&3.1.10.
	
	\begin{thm}\label{smooth}
		Let $Q$ be a finitely generated monoid, let $\sA_Q:=\Spec(Q\to R[Q])$ and let $S:=\Spec (0\to R)$.
		Then the following conditions are equivalent:
		\\
		$\bullet$ The order of the torsion subgroup of $Q^{\gp}$ is invertible in $R$.
		\\
		$\bullet$  The morphism of log schemes $\sA_Q\to S$ is smooth.
		\\
		$\bullet$  The group scheme $\sA^*_Q:=\Spec R[Q^{\gp}]$ is smooth over $S$.
	\end{thm}
	
	\begin{thm}\label{etale}
		Let $\theta:P\to Q$ be a morphism of finitely generated monoids,
		and let $f:\sA_Q\to \sA_P$ be the corresponding morphism of log schemes over a base ring $R$.
		Then the following conditions are equivalent:
		\\
		$\bullet$ The kernel and cokernel of $\theta^{gp}$ are finite groups whose order is invertible in $R$.
		\\
		$\bullet$ The morphism of log schemes $f:\sA_Q\to \sA_P$ is \'{e}tale over $R$.
		\\
		$\bullet$ The morphism of group schemes $f|_{\sA^*_Q}:\sA^*_Q\to \sA^*_P$ is \'{e}tale over $R$.
	\end{thm}
	
	An immediate corollary of Theorem \ref{smooth} is that 
	any toric variety $\Spec k[Q]$ over a characteristic $0$ field $k$, with the log structure associated to the natural map $Q\to k[Q]$, is smooth over $k$ with trivial log structure.
	
	\section{Local coordinates}
	
	\subsection{Coordinates on infinitesimal disc}
	Let $D^2$ be the (formal) log scheme associated to the (formal) log ring
	$\alpha_{D^2}:\bN^2\to\bC[[x,y]]$.
	The log structure is given by
	$$\bN^2\oplus\left\{
	a_0+\sum_{i\geq0,j\geq0,ij\not=0}a_{ij}x^iy^j:a_0\not=0
	\right\}=\bN^2\oplus\bC[[x,y]]^*\to\bC[[x,y]],\quad ((m,n),f)\mapsto x^my^n f.$$
	The underlying scheme $\ul D^2$ consists of four base points:
	$(x),(y),(x,y),0$.
	The ring $\bC[[x,y]]$ is a complete topological $\bC$-algebra, endowed with the basis 
	$x^my^n\bC[[x,y]]$ ($m,n\geq0$)
	of neighborhoods near $0$.
	
	\begin{definition}
		A \textit{coordinate transformation} of $D^2$ is a continuous automorphism of $D^2$ preserving all base points.
	\end{definition}
	
	A coordinate transformation $\rho$ is determined by its action on the topological generators $x$ and $y$, and
	thus it can be represented by
	$
	\left(
	\begin{matrix}
		\rho(x)\\
		\rho(y)
	\end{matrix}
	\right)
	$
	for some $\rho(x)$, $\rho(y)\in\bC[[x,y]]$.
	
	\begin{prop}\label{torsor}
		Let $\Aut^0D^2$ be the collection of all coordinate transformations of $D^2$.
		Then
		$$\Aut^0D^2\simeq
		\left\{
		\left(
		\begin{matrix}
			\sum_{i,j\geq0} a_{ij}x^iy^j\\
			\sum_{i,j\geq0} a'_{ij}x^iy^j
		\end{matrix}
		\right)
		: a_{10}\not=0,a'_{01}\not=0, a_{00}=a'_{00}=a_{0i}=a'_{i0}=0\ \text{for all}\ i
		\right\}.$$
	\end{prop}
	
	\begin{proof}
		Let $\rho(x)=\sum_{i,j\geq0} a_{ij}x^iy^j$, 
		$\rho(y)=\sum_{i,j\geq0} a'_{ij}x^iy^j$ be a coordinate transformation of $D^2$.
		It follows immediately that $a_{00}=a'_{00}=0$; otherwise any nonzero base point would be mapped to the total ring.
		And clearly if $a_{0i}\not=0$ for some $i$, the base point $(x)$ would not be preserved.
		By the same argument, 
		we obtain $a'_{i0}=0$ for all $i$.
		
		Our task now is to show that $a_{10}\not=0$ and $a'_{01}\not=0$.
		Let $\theta(x)$ and $\theta(y)$ be the inverses of $\rho(x)$ and $\rho(y)$ respectively,
		written as
		$\theta(x)=\sum b_{ij}x^iy^j$, $\theta(y)=\sum b'_{ij}x^iy^j$.
		Denote by $\tilde{y}_k=\sum_l b_{kl}y^l$ and $\bar{y}_k=\sum_l b'_{kl}y^l$.
		Then $\tilde{y}_0=0$, and we have
		\begin{equation*}
			\begin{aligned}
				\rho\circ\theta(x)&=
				\sum_{i,j}a_{ij}(\tilde{y}_0+\tilde{y}_1x+O(x^2))^i(\bar{y}_0+\bar{y}_1x+O(x^2))^j\\
				&=\sum_{i,j}a_{ij}(\tilde{y}_0^i\bar{y}_0^j+(j\tilde{y}_0^i\bar{y}_0^{j-1}\bar{y}_1+i\tilde{y}_0^{i-1}\tilde{y}_1\bar{y}_0^j)x+O(x^2)).
			\end{aligned}
		\end{equation*}
		The coefficient of $x$ in $\rho\circ\theta(x)$ is 
		$$\operatorname{coef} x=\sum_{i,j}a_{ij}(j\tilde{y}_0^i\bar{y}_0^{j-1}\bar{y}_1+i\tilde{y}_0^{i-1}\tilde{y}_1\bar{y}_0^j).$$
		Since $a_{0i}=0$ and $\tilde{y}_0=0$, 
		we see that
		$$\operatorname{coef} x=\sum_j a_{1j}\tilde{y}_1\bar{y}_0^j
		=\sum_j a_{1j}\left(\sum_l b_{1l}y^l\right)\left(\sum_lb'_{0l}y^l\right)^j$$
		whose constant term should be
		$$\sum_j a_{1j}b_{10}{b'_{00}}^j=a_{10}b_{10}=1.$$
		Thus $a_{10}\not=0$, 
		and $b_{10}\not=0$.
		Furthermore, the coefficient of the linear term in $y$ in $\operatorname{coef}x$ is supposed to be
		$$a_{10}b_{11}+a_{11}b_{10}{b'_{01}}=0$$
		which determines $b_{11}$ (since $b_{10}\not=0$) 
		and all other coefficients can be obtained step by step. 
		The proof for $\rho(y)$ and $\theta(y)$ follows analogously.
		
		Finally, we incorporate the assertions into log pattern.
		For $x^my^n\in\bC[[x,y]]$, its image under the coordinate transformation is $\rho(x)^m\rho(y)^n$.
		By the above discussion, there exists $g_{m,n}\in\bC[[x,y]]^*$ such that $\rho(x)^m\rho(y)^n=x^my^ng_{m,n}$.
		Now we give the corresponding log ring morphism:
		$$
		\begin{tikzcd}
			{\bC[[x,y]]} \arrow[r, "\rho^\sharp"]                                     & {\bC[[x,y]]}                                       &  & x^my^nf \arrow[r, maps to]                        & \rho(x)^m\rho(y)^n\rho(f)                   \\
			{\bN^2\oplus\bC[[x,y]]^*} \arrow[u, "\alpha_{D^2}"] \arrow[r, "\rho^\flat"] & {\bN^2\oplus\bC[[x,y]]^*} \arrow[u, "\alpha_{D^2}"'] &  & {((m,n),f)} \arrow[u, maps to] \arrow[r, maps to] & {((m,n),g_{m,n}\rho(f))} \arrow[u, maps to]
		\end{tikzcd}
		$$
		Moreover the monoid morphism $\bar{\rho^\flat}:\bN^2\to\bN^2$ induced by  $\rho^\flat$ is the identity map,
		and hence $\rho^\flat$ is strict by definition.
		Thus $\rho$ is an isomorphism of log rings,
		which completes the proof.
	\end{proof}

	\subsection{Coordinate changes for $\beta\gamma-bc$ system}\label{bgbc}
	
	\begin{definition}[\cite{frenkelvertex}-2.2.2]
		The \textit{normal ordered product} of two fields $A(z)=\sum_{n\in\bZ} A_{(n)}z^{-n-1}$, $B(w)=\sum_{n\in\bZ} B_{(n)}w^{-n-1}$
		on $V$ is defined as the formal power series
		$$:A(z)B(w):=A(z)_+B(w)+B(w)A(z)_-$$
		where for a formal power series $f(z)=\sum_{n\in\bZ}f_nz^n$, we write
		$$f(z)_+=\sum_{n\geq0}f_nz^n,\quad f(z)_-=\sum_{n<0}f_nz^n.$$
		The \textit{normal ordered product} 
		$:ab:$ of two elements $a,b\in V$
		is then defined to be
		$a_{(-1)}b$.
	\end{definition}
	
	\begin{rk}\label{regope}
		A basic and important property of normal ordered products of fields is that 
		$:A(z)B(w):$ 
		is regular in $z-w$.
		In fact,
		for any $v\in V$ and $\vp\in V^*=\Hom_\bC(V,\bC)$,
		we have that
		\begin{equation*}
			\begin{aligned}
			    &\langle
				\vp,:A(z)B(w):v
				\rangle
				\\
				&=
				\langle
				\vp,
				\sum_{n\in\bZ}\left(
				\sum_{m<0}A_{(m)}B_{(n)}vz^{-m-1}
				+\sum_{m\geq0}B_{(n)}A_{(m)}vz^{-m-1}
				\right)
				\rangle
				w^{-n-1}
				\\
				&=
				\sum_{n\in\bZ}\sum_{m<0}\langle
				\vp,A_{(m)}B_{(n)}v\rangle
				z^{-m-1}w^{-n-1}
				+
				\sum_{n\in\bZ}\sum_{m\geq0}\langle
				\vp,B_{(n)}A_{(m)}v\rangle
				z^{-m-1}w^{-n-1}
			\end{aligned}
		\end{equation*}
	    It is easy to see that both two terms of summation belong to
	    $\bC[[z,w]][z^{-1},w^{-1}]$.
	    This is equivalent to say that
	    $\langle
	    \vp,:A(z)B(w):v
	    \rangle$
	    is regular in $z-w$.
	\end{rk}
	
	The $\beta\gamma-bc$ \textit{system} is a conformal vertex algebra whose generating fields are even fields $\beta(z)$, $\gamma(z)$ and odd fields $b(z)$, $c(z)$, 
	with nontrivial OPEs:
	$$\beta(z)\gamma(w)=\frac{1}{z-w}+\operatorname{reg.}\sim\frac{1}{z-w},
	\quad 
	b(z)c(w)=\frac{1}{z-w}+\operatorname{reg.}\sim\frac{1}{z-w}$$

	Let us explain the above statements:
	\\
	(1)
	The vector space of the $\beta\gamma-bc$ system is spanned by elements of the form
	$$\beta_{k_1}...\beta_{k_s}\gamma_{n_1}...\gamma_{n_r}b_{l_1}...b_{l_t}c_{m_1}...c_{m_u}\0,$$
	where $\0$ is the vacuum vector such that 
	$$\beta_{n\geq0}\0=\gamma_{n>0}\0=b_{n\geq0}\0=c_{n>0}\0=0.$$
	We sometimes omit the vacuum vector when we  write a Fourier coefficient acting on it (for example, write $\gamma_{-1}$ rather than $\gamma_{-1}\0$) if no confusion arises.
	\\
	(2)
	We set
	$$\gamma(z)=\gamma_0(z)=\sum_{n\in\bZ}\gamma_nz^{-n},\quad
	\beta(z)=\beta_{-1}(z)=\sum_{n\in\bZ}\beta_nz^{-n-1},$$
	$$
	c(z)=c_0(z)=\sum_{n\in\bZ}c_nz^{-n},\quad
	b(z)=b_{-1}(z)=\sum_{n\in\bZ}b_nz^{-n-1}
	$$
	and the vertex operation is induced by
	$$\gamma_{-n}(z)=\partial_z^{(n)}\gamma(z),\quad
	\beta_{-n-1}(z)=\partial_z^{(n)}\beta(z),\quad
	c_{-n}(z)=\partial_z^{(n)}c(z),\quad
	b_{-n-1}(z)=\partial_z^{(n)}b(z)
	$$
	for $n>0$ where $\partial^{(n)}_z=\frac{1}{n!}\partial^n_z$.
	\\
	(3)
	The translation operator is $T=L_{-1}$
	where
	$$L(z)=\sum_{n\in\bZ}L_nz^{-n-2}=:\partial_z\gamma(z)\beta(z):+:\partial_z c(z)b(z):.$$
	(4) ``$\operatorname{reg.}$'' in OPEs denotes some regular formal power series in $z-w$.
	The regular terms do not contribute to the relation of endomorphisms,
	and so they are usually omitted. 
	In particular, the nontrivial relations are
	\begin{equation}\label{bracket}
		[\beta_m,\gamma_n]=\delta_{m,-n},\quad
		[b_m,c_n]_+=\delta_{m,-n}.
	\end{equation}
	And we also have that
	$$\gamma(z)\beta(w)\sim\frac{-1}{z-w}
	,\quad
	c(z)b(w)\sim\frac{1}{z-w}.$$
	
	We can verify that the above setting satisfies the following reconstruction theorem, and hence the $\beta\gamma-bc$ system is a vertex algebra as we claimed above.
	
	\begin{thm}[\cite{frenkelvertex}-2.3.10]
		Let $V$ be a vector space, $\0$ a non-zero vector, and $T$ an endomorphism of $V$.
		Let $S$ be a countable ordered set and $\{a^\alpha:\alpha\in S\}$ a collection of vectors in $V$.
		Suppose we are also given fields $a^\alpha(z)=\sum_{n\in\bZ}a^\alpha_{(n)}z^{-n-1}$ such that the following conditions hold:
		\\
		$\bullet$ For all $\alpha$, $a^\alpha(z)\0\in V[[z]]$;
		\\
		$\bullet$ $T\0=0$ and $[T,a^\alpha(z)]=\partial_z a^\alpha(z)$ for all $\alpha$;
		\\
		$\bullet$ For any pair of fields $a^{\alpha_1}(z)$, $a^{\alpha_2}(z)$, there exists $N\in\bZ_{>0}$ such that $(z-w)^N[a^{\alpha_1}(z),a^{\alpha_2}(w)]=0$ as a formal power series in $(\End V)[[z^{\pm1},w^{\pm1}]]$.
		\\
		$\bullet$ $V$ has a basis of vectors $a^{\alpha_1}_{(j_1)}...a^{\alpha_m}_{(j_m)}\0$
		where $j_1\leq ...\leq j_m<0$, and if $j_i=j_{i+1}$ then $\alpha_i\leq \alpha_{i+1}$ with respect to th given order on $S$.
		Then the assignment
		$$Y\left(a^{\alpha_1}_{(j_1)}...a^{\alpha_m}_{(j_m)}\0,z\right)=:\partial^{(-j_1-1)}_za^{\alpha_1}(z)...\partial^{(-j_m-1)}_za^{\alpha_m}(z):$$
		defines a vertex algebra structure on $V$. 
		Moreover if $V$ is a $\bZ$-graded vector space,
		$\deg\0=0$, 
		the vectors $a^\alpha$ are homogeneous, $\deg T=1$, and the fields $a^\alpha(z)$ have conformal dimension $\deg a^\alpha$, then $V$ is a $\bZ$-graded vertex algebra.
	\end{thm}

	Let $D$ be the formal log scheme $\Spf\bC[[\gamma]]$,
	endowed with the log structure associated to
	$\bN\to\bC[[\gamma]]$, $n\mapsto \gamma^n$.
	Following \cite{malikov1999chiral}-3.6, consider the formal $1|1$-dimensional superscheme $\tilde{D}=\Pi TD$ where
	$TD$ is the total space of the tangent bundle over $D$ and $\Pi$ is the parity change functor.
	Then the underlying topological space of $\tilde{D}$ is the same as that of $D$, i.e. a single point, 
	and the structure sheaf $\cO_{\tilde{D}}$ is isomorphic to the de Rham algebra of differential forms on $D$.
	Concretely $\tilde{D}$ admits an even coordinate $\gamma$ and an odd coordinate $c=d1=d\alpha_{D^2}(1)/\alpha_{D^2}(1)=d\gamma/\gamma$ (the log differential)
	for $1\in\bN$.
	Geometrically, the fields $\beta$'s (resp. $b$'s) correspond to the vector fields $\partial_\gamma$'s (resp. $\partial_c$'s).
	
	\begin{rk}
		The geometric interpretation of $\beta,\gamma, b, c$ in log geometry follows the same idea as in the classical case,
		which satisfies the relation (\ref{bracket}).
		This idea suggests how to define the coordinate change formulas for the fields $\beta(z),b(z),c(z)$:
		they should satisfy the same relations as in the classical case,
		as shown in the next theorem.
	\end{rk}
	
	Let $f$ be a coordinate transformation of $D$ and $g$ its inverse.
	We denote by $\tilde{\gamma}=f(\gamma)$, $\gamma=g(\tilde{\gamma})$, 
	and use a tilde above a vector to denote the coordinate changed one.
	After a tedious computation (referring to \cite{malikov1999chiral}-3.6) we yield the following coordinate changes of the generating fields of the $\beta\gamma-bc$ system (in log setting):
	\begin{equation}\label{trans}
		\begin{aligned}
			\tilde{c}&=\frac{df(\gamma)}{f(\gamma)}=\frac{\gamma\partial_\gamma f(\gamma)}{f(\gamma)}c,
			\\
			\tilde{b}&=f(\gamma)\partial_{\tilde{\gamma}}g(\tilde{\gamma})|_{\tilde{\gamma}=f(\gamma)}\partial_\gamma=\frac{f(\gamma)}{\gamma}\partial_{\tilde{\gamma}}g(\tilde{\gamma})|_{\tilde{\gamma}=f(\gamma)}b,
			\\
			\tilde{\beta}&=\partial_{\tilde{\gamma}}g(\tilde{\gamma})|_{\tilde{\gamma}=f(\gamma)}\partial_\gamma +\partial_{\tilde{\gamma}}(
			\frac{\tilde{\gamma}\partial_{\tilde{\gamma}}g(\tilde{\gamma})}{g(\tilde{\gamma})}\tilde{c})|_{\tilde{\gamma}=f(\gamma)}\partial_c
			\\
			&=\partial_{\tilde{\gamma}}g(\tilde{\gamma})|_{\tilde{\gamma}=f(\gamma)}\beta
			+\partial_{\tilde{\gamma}}^2g(\tilde{\gamma})|_{\tilde{\gamma=f(\gamma)}}\partial_\gamma f(\gamma)cb
			+\partial_{\tilde{\gamma}}(\frac{\tilde{\gamma}}{g(\tilde{\gamma})})\partial_{\tilde{\gamma}}g(\tilde{\gamma})|_{\tilde{\gamma}=f(\gamma)}\frac{\gamma\partial_{\gamma}f(\gamma)}{f(\gamma)}cb.
		\end{aligned}
	\end{equation}
	Due to \cite{frenkelvertex}-6.2,
	the coordinate transformation $f$ can be represented by
	$$f(\gamma)=a_1\gamma+a_2\gamma^2+...\in\bC[[\gamma]]$$ 
	with $a_1\not=0$
	and thus $f(\gamma)/\gamma$ is a unit in $\bC[[\gamma]]$.
	This also implies that $\gamma/f(\gamma)\in\bC[[\gamma]]$.
	Together with \cite{malikov1999chiral}-3.1,
	we conclude that vectors $\tilde{\gamma}$, $\tilde{c}$, $\tilde{b}$ and $\tilde{\beta}$ are well-defined.
	
	Therefore we obtain the coordinate changes of the corresponding fields:
	\begin{equation}\label{tilde}
		\begin{aligned}
			&\tilde{\gamma}(z)=f(\gamma)(z),
			\\
			&\tilde{c}(z)=\frac{\gamma\partial_\gamma f(\gamma)}{f(\gamma)}(z)c(z),
			\\
			&\tilde{b}(z)=:\frac{f(\gamma)}{\gamma}\partial_{\tilde{\gamma}}g(\tilde{\gamma})|_{\tilde{\gamma}=f(\gamma)}(z)b(z):,
			\\
			&
			\begin{aligned}
				\tilde{\beta}(z)
				=:\partial_{\tilde{\gamma}}g(\tilde{\gamma})|_{\tilde{\gamma}=f(\gamma)}(z)\beta(z):&
				+::\partial_{\tilde{\gamma}}^2g(\tilde{\gamma})|_{\tilde{\gamma}=f(\gamma)}\partial_\gamma f(\gamma)(z)c(z):b(z):
				\\
				&+::\partial_{\tilde{\gamma}}(\frac{\tilde{\gamma}}{g(\tilde{\gamma})})\partial_{\tilde{\gamma}}g(\tilde{\gamma})|_{\tilde{\gamma}=f(\gamma)}\frac{\gamma\partial_{\gamma}f(\gamma)}{f(\gamma)}(z)c(z):b(z):.
			\end{aligned}
		\end{aligned}
	\end{equation}
	
	\begin{thm}\label{change}
		The fields $\tilde{\gamma}(z)$, $\tilde{c}(z)$, $\tilde{b}(z)$ and $\tilde{\beta}(z)$ satisfy the following relations:
		\begin{equation*}
			\begin{aligned}
				\tilde{\beta}(z)\tilde{\gamma}(w)&\sim\frac{1}{z-w},\quad
				\tilde{c}(z)\tilde{b}(w)\sim\frac{1}{z-w},
				\\
				A(z)B(w)&\sim0\quad
				\text{for all}\ A,B=\tilde{\gamma}, \tilde{c}, \tilde{b}, \tilde{\beta}\
				\text{but the above two cases.}
			\end{aligned}
		\end{equation*}
	\end{thm}
	
	\begin{proof}
		The nontrivial relations are
		$\tilde{c}(z)\tilde{b}(w)$, $\tilde{b}(z)\tilde{\beta}(w)$ and $\tilde{c}(z)\tilde{\beta}(w)$,
		and the others are clear or exactly the same to the classical results in \cite{malikov1999chiral}-3.6.
		Since $\gamma$ commutes with $c$,
		we have that
		\begin{equation*}
			\begin{aligned}
				\tilde{c}(z)\tilde{b}(w)&=\frac{\gamma\partial_\gamma f(\gamma)}{f(\gamma)}(z)c(z)
				:\frac{f(\gamma)}{\gamma}\partial_{\tilde{\gamma}}g(\tilde{\gamma})|_{\tilde{\gamma}=f(\gamma)}(w)b(w):\\
				&=\partial_\gamma f(\gamma)(z)c(z)
				\partial_{\tilde{\gamma}}g(\tilde{\gamma})|_{\tilde{\gamma}=f(\gamma)}(w)_+b(w)+
				\partial_\gamma f(\gamma)(z)c(z)b(w)
				\partial_{\tilde{\gamma}}g(\tilde{\gamma})|_{\tilde{\gamma}=f(\gamma)}(w)_-\\
				&\sim\frac{1}{z-w}\quad \text{by the exactly same proof in \cite{malikov1999chiral}-Theorem 3.7}.
			\end{aligned}
		\end{equation*}
		The proofs of $\tilde{b}(z)\tilde{\beta}(w)$ and $\tilde{c}(z)\tilde{\beta}(w)$ are quite similar, and so we only provide the former.
		We decompose  $\tilde{b}(z)\tilde{\beta}(w)=\mathfrak{1}+\mathfrak{2}$
		where $\mathfrak{1}$ (resp. $\mathfrak{2}$) is the product of $\tilde{b}(z)$ and the first (resp. second and third) term of $\tilde{\beta}(z)$ in \eqref{tilde}.
		The following two relations will be used, 
		referring to \cite{malikov1999chiral}-(3.18):
		$$
		h(\gamma)(z)\beta(w)\sim-\frac{\partial_{\gamma}h(w)}{z-w},\quad
		\beta(z)h(\gamma)(w)\sim
		\frac{\partial_{\gamma}h(w)}{z-w}$$
		for each formal power series $h$ over $\bC$.
		For the first term of $\tilde{b}(z)\tilde{\beta}(w)$, we have
		\begin{equation*}
			\begin{aligned}
				\mathfrak{1}&=
				:\frac{f(\gamma)}{\gamma}\partial_{\tilde{\gamma}}g(\tilde{\gamma})|_{\tilde{\gamma}=f(\gamma)}(z)b(z):
				:\partial_{\tilde{\gamma}}g(\tilde{\gamma})|_{\tilde{\gamma}=f(\gamma)}(w)\beta(w):\\
				&
				=\frac{f(\gamma)}{\gamma}\partial_{\tilde{\gamma}}g(\tilde{\gamma})|_{\tilde{\gamma}=f(\gamma)}(z)b(z)
				\left(
				\partial_{\tilde{\gamma}}g(\tilde{\gamma})|_{\tilde{\gamma}=f(\gamma)}(w)_+\beta(w)
				+\beta(w)\partial_{\tilde{\gamma}}g(\tilde{\gamma})|_{\tilde{\gamma}=f(\gamma)}(w)_-
				\right)
				\\
				&
				=b(z)\partial_{\tilde{\gamma}}g(\tilde{\gamma})|_{\tilde{\gamma}=f(\gamma)}(w)_+\frac{f(\gamma)}{\gamma}\partial_{\tilde{\gamma}}g(\tilde{\gamma})|_{\tilde{\gamma}=f(\gamma)}(z)\beta(w)
				+
				b(z)\frac{f(\gamma)}{\gamma}\partial_{\tilde{\gamma}}g(\tilde{\gamma})|_{\tilde{\gamma}=f(\gamma)}(z)\beta(w)\partial_{\tilde{\gamma}}g(\tilde{\gamma})|_{\tilde{\gamma}=f(\gamma)}(w)_-
				\\
				&
				\sim -\frac{1}{z-w}\partial_\gamma
				\left(
				\frac{f(\gamma)}{\gamma}
				\partial_{\tilde{\gamma}}g(\tilde{\gamma})|_{\tilde{\gamma=f(\gamma)}}
				\right)
				\partial_{\tilde{\gamma}}g(\tilde{\gamma})|_{\tilde{\gamma=f(\gamma)}}(w)
				b(z)
				\\
				&
				\sim -\frac{1}{z-w}
				:(\partial_\gamma f(\gamma))
				\left(
				\partial_{\tilde{\gamma}}\left(\frac{\tilde{\gamma}}{g(\tilde{\gamma})}\right)|_{\tilde{\gamma=f(\gamma)}}\partial_{\tilde{\gamma}}g(\tilde{\gamma})|_{\tilde{\gamma=f(\gamma)}}
				+
				\frac{\tilde{\gamma}}{g(\tilde{\gamma})}
				\partial^2_{\tilde{\gamma}}g(\tilde{\gamma})|_{\tilde{\gamma=f(\gamma)}}
				\right)
				\partial_{\tilde{\gamma}}g(\tilde{\gamma})|_{\tilde{\gamma=f(\gamma)}}
				b:(w)
			\end{aligned}
		\end{equation*}
		where the first ``$\sim$''  is due to 
		$\partial_\gamma=(\partial_\gamma f(\gamma))\partial_{\tilde{\gamma}}$,
		and the second ``$\sim$'' follows from the formal Taylor formula $b(z)=b(w)+(z-w)\partial_wb(w)+...$
		(\cite{frenkelvertex}-3.2.4).
		For the second and third terms,
		by the Wick theorem 
		(\cite{Kac}-Theorem 3.3)
		we obtain that
		\begin{equation*}
			\begin{aligned}
				\mathfrak{2}\sim\frac{1}{z-w}
				:\frac{\tilde{\gamma}}{g(\tilde{\gamma})}
				\partial_{\tilde{\gamma}}g(\tilde{\gamma})|_{\tilde{\gamma}=f(\gamma)}
				\Bigg(&
				\partial^2_{\tilde{\gamma}}g(\tilde{\gamma})|_{\tilde{\gamma}=f(\gamma)}\partial_{\gamma}f(\gamma)\\
				&+
				\partial_{\tilde{\gamma}}
				\left(\frac{\tilde{\gamma}}{g(\tilde{\gamma})}\right)
				|_{\tilde{\gamma}=f(\gamma)}
				\partial_{\tilde{\gamma}}g(\tilde{\gamma})|_{\tilde{\gamma}=f(\gamma)}
				\frac{\partial_{\gamma}f(\gamma)\cdot g(\tilde{\gamma})}{\tilde{\gamma}}|_{\tilde{\gamma}=f(\gamma)}
				\Bigg)
				b:(w).
			\end{aligned}
		\end{equation*}
		Then it is direct to see that  $\mathfrak{1}$ and $\mathfrak{2}$ cancel out,
		and hence $\tilde{b}(z)\tilde{\beta}(w)$  is regular.
		Consequently the fields under coordinate changes  have the same OPEs as the original ones.
	\end{proof}
	
	\begin{rk}\label{Q0}
		Let $Q=\beta_{-1}c_0$,
		which plays the role of chiral de Rham differential in the classical case
		and will be discussed in Section 4.
		We work with $\tilde{Q}(z)=:\tilde{\beta}(z)\tilde{c}(z):$.
		Let
		$$
		A(\gamma)=\partial_{\tilde{\gamma}}g(\tilde{\gamma})|_{\tilde{\gamma}=f(\gamma)},
		\quad
		B(\gamma)=\partial_{\tilde{\gamma}}^2g(\tilde{\gamma})|_{\tilde{\gamma=f(\gamma)}}\partial_\gamma f(\gamma),
		\quad
		C(\gamma)=\partial_{\tilde{\gamma}}(\frac{\tilde{\gamma}}{g(\tilde{\gamma})})\partial_{\tilde{\gamma}}g(\tilde{\gamma})|_{\tilde{\gamma}=f(\gamma)}\frac{\gamma\partial_{\gamma}f(\gamma)}{f(\gamma)}.
		$$
		Denote by $h(\gamma)=\frac{\gamma \partial_{\gamma}f(\gamma)}{f(\gamma)}$,
		and then we have 
		\begin{equation*}
			\begin{aligned}
				\tilde{Q}(z)=
				::A(\gamma)(z)\beta(z)::h(\gamma)(z)c(z)::
				+::(B(\gamma)(z)+C(\gamma)(z))c(z)b(z):h(\gamma)(z)c(z):
			\end{aligned}
		\end{equation*}
		Since $\gamma$ commutes with $c$ and the fact that $c$ is odd,
		it follows that the third term
		$$::C(\gamma)(z)c(z)b(z):
		h(\gamma)(z)c(z):=
		:C(\gamma)(z)h(\gamma)(z)c(z):.
		$$
		Now we deal with the first two terms
		\begin{equation*}
			\begin{aligned}
				::A(\gamma)(z)\beta(z)::h(\gamma)(z)c(z)::
				+
				::B(\gamma)(z)c(z)b(z):
				h(\gamma)(z)c(z):
			\end{aligned}
		\end{equation*}
		which is almost of the same form with that in the classical case.
		Indeed, comparing this with (4.2) in \cite{malikov1999chiral}-Theorem 4.2,
		one observes that if we replace 
		$\frac{\partial g}{\partial b}$ by $h(\gamma)$,
		it follows exactly that
		$$
		\tilde{Q}(z)=
		:\frac{\gamma}{f(\gamma)}Q:(z)+
		\partial_z
		(:\partial^2_{\tilde{\gamma}}g(\tilde{\gamma})|_{\tilde{\gamma=f(\gamma)}}h(\gamma)\tilde{c}:(z))
		+:C(\gamma)h(\gamma)c:(z).$$
	\end{rk}

	\begin{rk}
		We give a brief explanation of the well-definedness of $\frac{f(\gamma)}{\gamma}(z)$ appearing in the above proof and remark.
		Let $V$ be the vertex algebra generated by even fields $\beta(z)$ and $\gamma(z)$ with the OPEs in the $\beta\gamma-bc$ system
		(this is the $\beta\gamma$-Heisenberg vertex algebra that we will introduce in the next section),
		and let $\hat{V}=\bC[[\gamma]]\otimes_{\bC[\gamma]}V$.
		Generally for an element $f\in\bC[[z]]$,
		the series $f(\gamma(z))$ is a well-defined element in $(\End\hat{V})[[z^{\pm1}]]$.
		Roughly speaking,
		if we write 
		$\gamma(z)=\gamma_0+\Delta(z)$,
		then $f(\gamma(z))$ is defined by Taylor formula:
		$$f(\gamma(z))=\sum_{i\geq0}\Delta(z)^i\partial_\gamma^{(i)}f(\gamma).$$
		Checking that the action of $f(\gamma(z))$ is well-defined,
		is rather technical and can be found in \cite{malikov1999chiral}-3.1.
	\end{rk}
	
	The above two results illustrate that the  vertex algebra structure on the $\beta\gamma-bc$ system is canonical
	while the conformal structure is not necessarily preserved.
	This fact will be used in the construction of chiral de Rham complex.
	
	\subsection{Rational singularities of $A_n$-type}\label{sing}
	
	Let $\epsilon\in\bC$ be a primitive $N$-th root of unity and $G$ a cyclic group generated by $g$ of order $N$.
	There is a $G$-action on the coordinate ring of the complex plane $\bC[x,y]$, associated to 
	$$g\sim 
	\left(
	\begin{matrix}
		\epsilon&0\\
		0&\epsilon^{-1}
	\end{matrix}
	\right).$$
	Explicitly, $\epsilon\cdot x=\epsilon x$ and
	$\epsilon\cdot y=\epsilon^{-1} y$.
	It is easy to see that the subring of $G$-invariants in $\bC[x,y]$ is 
	$$\bC[x,y]^G=\bC[x^N,y^N,xy].$$
	The group $G$ is a finite cyclic subgroup of $SL_2(\bC)$,
	and thus the scheme $\ul\sA_N:=\Spec\bC[x,y]^G$
	is a surface with a singularity of $A_N$-type at the origin.
	Discussions about rational singularities of other types can be found in \cite{burban2019val}-3.
	
	Let $Q$ be the submonoid of $\bN^2$ generated by $(N,0)$, $(0,N)$ and $(1,1)$.
	Then the monoid morphism
	$Q\to\bC[x,y]^G$ given by
	$(N,0)\mapsto x^N$, $(0,N)\mapsto y^N$, $(1,1)\mapsto xy$
	defines a log ring.
	We denote by $\sA_N$ the log scheme
	associated to the above log ring.
	Clearly we have $\sA_N\simeq \sA_Q$, and hence by Theorem \ref{smooth}, 
	$\sA_N$ is (log) smooth over $\bC$.
	
	\subsection{Local coordinates on surface}
	\begin{definition}
		A \textit{local coordinate} on $\sA_N$ is a formally \'{e}tale morphism of log schemes $D^2\to\sA_N$.
	\end{definition}
	
	According to \cite{ogus2018lectures}-3.1.6, log smoothness and usual smoothness are equivalent outside the singularities of the base scheme.
	Thus the usual notion of local coordinates agrees with the above one in log pattern.
	We now provide a basic local coordinate at the singularity.
	Consider the morphism of log rings
	\begin{equation}\label{coor}
		\begin{tikzcd}
			{\bC[x,y]^G} \arrow[r, hook] & {\bC[[x,y]]}    \\
			Q \arrow[r, hook] \arrow[u]  & \bN^2 \arrow[u]
		\end{tikzcd}
	\end{equation}
	which induces a morphism of log schemes $\phi:D^2\to\sA_N$.
	The image of $|\ul\phi|$ is the origin of $|\ul\sA_N|$.
	We claim that $\phi$ defines a local coordinate at the singularity.
	Diagram (\ref{coor}) can be decomposed into 
	\begin{equation}
		\begin{tikzcd}
			{\bC[x,y]^G} \arrow[r, "\hat{\phi}^\sharp"] & {\bC[x,y]} \arrow[r] & {\bC[[x,y]]}               \\
			Q \arrow[u] \arrow[rr, hook]         &                      & \bN^2 \arrow[u] \arrow[lu]
		\end{tikzcd}
	\end{equation}
	and then it leads to morphisms of log schemes
	$D^2\to \sA_{\bN^2}\to \sA_N$.
	The monoid morphism $\hat{\phi}^\flat$ associated to $\hat{\phi}^\sharp$ coincides with $\phi^\flat$, i.e. the inclusion $Q\hookrightarrow\bN^2$.
	We apparently have that
	$\Ker \hat{\phi}^\flat=0$ and
	$\operatorname{Coker}\hat{\phi}^\flat\simeq \bZ/N\bZ$.
	Then it follows that $\hat{\phi}:\sA_{\bN^2}\to\sA_N$ is \'{e}tale from 
	Theorem \ref{etale}.
	The morphism $D^2\to \sA_{\bN^2}$ is formally \'{e}tale since any $\bC$-homomorphism with domain $\bC[x,y]$ or $\bC[[x,y]]$ is determined by the images of $x$ and $y$.
	Therefore the composition $\phi:D^2\to \sA_N$ is formally \'{e}tale as well.
	According to Proposition \ref{torsor}, any local coordinate arises from a coordinate transformation of $\phi$. 
	
	In summary, if we let $\Aut_x$ be the space of all local coordinates at a geometric point $x\in\sA_N$,
	i.e. the origin $(x,y)$ of $D^2$ is mapped to $x$, 
	then the $\Aut^0D^2$-action is transitive on $\Aut_x$.
	And the space of coordinates on $\sA_N$:
	$$\Aut_{\sA_N}=\{
	(x,\phi_x):x\in\sA_N, \phi_x\in \Aut_x
	\}$$
	is an $\Aut^0D^2$-torsor over $D^2$.

	\section{Chiral de Rham complex}\label{cdr}
	
	\subsection{Notations}
	The vertex algebra $\Omega_N$ is simply defined as the tensor product of $N$ copies of the $\beta\gamma-bc$ system as described in \cite{song2016global}.
	We provide a more detailed description here, referring to \cite{frenkel2003chiral}.
	
	Let $N$ be a positive integer.
	We denote by $H_N$ and $\operatorname{Cl}_N$ the infinite-dimensional Lie (super)algebras generated by
	even elements $\beta^i_n$, $\gamma^i_n$, and
	odd elements $b^i_n$, $c^i_n$ ($i=1,...,N$, $n\in\bZ$) respectively, with nontrivial Lie brackets:
	$$[\beta^i_m,\gamma^j_n]=\delta_{i,i}\delta_{n,-m},\quad
	[b^i_m,c^j_n]_+=\delta_{i,i}\delta_{n,-m}.
	$$
	Let $H_N^+$ be the Lie subalgebra of $H_N$ generated by $\beta_n^i$, $\gamma_m^i$ ($i=1,...,N$, $n\geq0$, $m>0$).
	And similarly let $\operatorname{Cl}_N^+$ be the Lie subalgebra of $\operatorname{Cl}_N$ generated by $b_n^i$, $c_m^i$ ($i=1,...,N$, $n\geq0$, $m>0$).
	The $\beta\gamma$-\textit{Heisenberg vertex algebra} $V_N$ and \textit{Clifford vertex algebra} $\Lambda_N$ are respectively defined as
	$$V_N=\Ind^{H_N}_{H_N^+}\bC,\quad 
	\Lambda_N=\Ind^{\operatorname{Cl}_N}_{\operatorname{Cl}_N^+}\bC$$
	where $\bC$ is the trivial representation of $H_N^+$ and $\operatorname{Cl}_N^+$.
	The vertex algebra $\Omega_N$ is defined to be
	$$\Omega_N=V_N\otimes \Lambda_N$$
	which is indeed conformal, with Virasoro element
	$$L=\sum_{i=1}^N \gamma_{-1}^i\beta_{-1}^i+c_{-1}^ib_{-1}^i.$$
	
	The vertex algebra $\Omega_N$ is graded by the \textit{fermionic charge} operator
	$$F=\sum_{i,n}:c^i_nb^i_{-n}:.$$
	It immediately follows that 
	$$F\0=0,\quad
	[F,c_n^i]=c_n^i,\quad
	[F,b_n^i]=-b_n^i,\quad
	[F,\gamma_n^i]=[F,\beta_n^i]=0.$$
	We denote that
	$$
	\Omega_N=\bigoplus_{p\in \bZ}\Omega^p_N,\quad
	\Omega^p_N:=\{\omega\in\Omega_N:F\omega=p\omega\}.$$
	Then $\Omega_N$ becomes a complex with respect to the \textit{chiral de Rham differential}
	$$d=\sum_{i,n}:\beta_n^ic_{-n}^i:.$$
	
	\begin{rk}
		Recall that we set
		$Q=\sum_{i=1}^N\beta_{-1}^ic_0^i$,
		then it follows from Remark \ref{Q0} that under the coordinate changes, 
		the chiral de Rham differential $d=Q_0$
		for a single $\beta\gamma-bc$ system becomes
		$\left(
		\left(
		\frac{\gamma}{f(\gamma)}
		\right)_0
		Q
		\right)_0+(C(\gamma)_0h(\gamma)_0c)_1$.
		Hence the chiral de Rham differential is not canonical,
		which differs from 
		its counterpart in the classical setting.
	\end{rk}
	
	\begin{rk}
		It is shown in \cite{malikov1999chiral}-3.1 that the vertex algebra structure on $V_N$ can be extended to $\widehat{V}_N=\bC[[\gamma_0^1,...,\gamma_0^N]]\otimes_{\bC[\gamma_0^1,...,\gamma_0^N]}V_N$.
		Thus we obtain a vertex algebra $\widehat{\Omega}_N=\widehat{V}_N\otimes\Lambda_N$,
		which contains the de Rham algebra of differential forms over $D^N=\Spf\bC[[\gamma_0^1,...,\gamma_0^N]]$.
	\end{rk}
	
	\begin{definition}
		The \textit{chiral de Rham complex} $\Omega^{\ch}_X$ over a scheme (or manifold) $X$ is defined to be the sheaf of vertex algebras associated to $\Omega_N$.
	\end{definition}
	
	There are two equivalent constructions of the sheaf structure in the usual smooth case.
	One of these (ref. \cite{malikov1999chiral}-3) is in the classical way via localizing $V_N$ (which is called the \textit{chiral structure sheaf}).
	The other one can be found in \cite{frenkel2003chiral}-3.4.
	Let $\Aut_X$ be the $\Aut^0\cO_N$-torsor of coordinates on $X$ (ref. \cite{frenkelvertex}-6.5).
	Then in fact the twist
	$$\tilde{\Omega}_X^{\ch}:=\Aut_X\times_{\Aut^0\cO_N}\widehat{\Omega}_N$$
	is a $\cD_X$-module,
	with a flat connection $\nabla$ locally induced by
	$\partial_{t_i}-\beta_0^i$ for local coordinates $t_1,...,t_N$.
	And the chiral de Rham complex $\Omega_X^{\ch}$ is defined to be the sheaf of horizontal sections of $\tilde{\Omega}_X^{\ch}$
	with respect to $\nabla$.
	
	\begin{rk}
		In the present paper we only consider surfaces and recall
		in \ref{bgbc} that $\gamma$'s and $c$'s are coordinates of the corresponding superscheme.
		And hence in sequel the $N$ in $\Omega_N$ is always taken to be $2$.
		Besides, we concentrate on global sections of the chiral de Rham complex.
		So we do not need the sheaf structure here.
		For convenience,
		we call the space of global sections of $\Omega_X^{\ch}$ the chiral de Rham complex as well if no confusion arises. 
	\end{rk}
	
	\subsection{Global sections of the chiral de Rham complex on $A_N$-singular surfaces}
	
	Let $Q$ be the submonoid of $\bN^2$ generated by $(N,0)$, $(0,N)$ and $(1,1)$ as before.
	Then the underlying scheme of $\sA_N\simeq\sA_Q$ is isomorphic to the toric variety $\Spec\bC[x,y,z]/(xy-z^N)$,
	which admits a singularity at the origin of $A_N$-type (recall \ref{sing}).
	A local coordinate $\phi$ at the origin on $\sA_N$ is given in diagram (\ref{coor}):
	$$
	\begin{tikzcd}
		{\frac{\bC[x,y,z]}{(xy-z^N)}} \arrow[r] & {\bC[[\gamma^1,\gamma^2]]} \\
		Q \arrow[r, hook] \arrow[u, "\alpha_N"] & \bN^2 \arrow[u]           
	\end{tikzcd}
	$$
	with the upper horizontal map given by
	$$x\mapsto (\gamma^1)^N,\quad
	y\mapsto(\gamma^2)^N,\quad
	z\mapsto\gamma^1\gamma^2.$$
	
	For a monoid $P$, we denote by $\pi:P\to P^{\gp}$ the natural map.
	For the log ring $\alpha_N:Q\to\bC[Q]\simeq\bC[x,y,z]/(xy-z^N)$, the module of log differentials is
	$$\Omega_Q\simeq \Omega_{\bC[Q]/\bC}\oplus(\bC[Q]\otimes Q^{\gp})/R$$
	where $R$ is a submodule of  $\Omega_{\bC[Q]/\bC}\oplus(\bC[Q]\otimes Q^{\gp})$ generated by $(d\alpha_N(q),-\alpha_N(q)\otimes\pi(q))$ for all $q\in Q$.
	We denote $\pi(q)$ by $dq$ for $q\in Q$.
	It is straightforward to see that $Q^{\gp}$ is a group of rank $2$, generated by $p=(1,1)$ and $q=(-1,1)$.
	Therefore the module $\Omega_Q$ is a $\bC[Q]$-module generated by $dp$, $dq$ and $dp\wedge dq$, i.e.
	$\Omega_Q\simeq \bC[Q]\otimes\wedge(p,q)$.
	It follows immediately that 
	$$dp\sim \frac{dz}{z},\quad dq\sim\frac{1}{N}
	\left(
	\frac{dy}{y}-\frac{dx}{x}
	\right).$$
	Using the local coordinate $\phi$, the log differentials can be transferred via $\phi^*\Omega_{\sA_N}\to \Omega_{D^2}$, with
	$$d(N,0)\sim\frac{dx}{x}\mapsto N\frac{d\gamma^1}{\gamma^1}
	,\quad
	d(0,N)\sim\frac{dy}{y}\mapsto N\frac{d\gamma^2}{\gamma^2},\quad
	d(1,1)\sim\frac{dz}{z}\mapsto \frac{d\gamma^1}{\gamma^1}+\frac{d\gamma^2}{\gamma^2}.
	$$
	Then we associate fields from log differentials with those on $D^2$ (recall that $c$'s correspond to log differential forms in section 3):
	$$\frac{dx}{x}(z)=Nc^1(z),\quad 
	\frac{dy}{y}(z)=Nc^2(z),\quad 
	\frac{dz}{z}=c^1(z)+c^2(z).$$
	Inserting the above fields to the generators of $\Omega_Q$,
	we have
	$$dp(z)=c^1(z)+c^2(z),\quad
	dq(z)=c^2(z)-c^1(z).$$ 
	Fields $dp(z)$, $dq(z)$ together with $${dp}^*(z)=\frac{1}{2}(b^1(z)+b^2(z)),\quad{dq}^*(z)=\frac{1}{2}(b^2(z)-b^1(z))
	$$
	which satisfy the following OPEs (the other OPEs are regular):
	$$
	dp^*(z)dp(w)\sim\frac{1}{z-w},\quad
	dq^*(z)dq(w)\sim\frac{1}{z-w},
	$$
	generate the Clifford vertex algebra $\Lambda_Q$ associated to $\sA_N$.
	Apparently $\Lambda_Q$ is isomorphic to $\Lambda_2$,
	as $b^i(z)$, $c^i(z)$, $i=1,2$ can be linearly generated by
	$dp(z)$, $dq(z)$, $dp^*(z)$, $dq^*(z)$.
	\begin{rk}
		The underlying scheme $\ul\sA_N$ is a toric variety with torus $\ul\sA_{Q^{\gp}}$.
		The inclusion of the torus $\ul\sA_{Q^{\gp}}$ into $\ul\sA_N$ corresponds to the natural ring homomorphism
		$\bC[Q]\to\bC[Q^{\gp}]$.
		According to \cite{ogus2018lectures}-3.4.1, we have $\operatorname{rank} Q^{\gp}=\dim\sA_N$.
		Moreover $\Omega_Q$ is generated by $Q^{\gp}$,
		with respect to the log differential $d$.
		So it is natural to take $dp$ and $dq$ as coordinates of the log differential module.
	\end{rk}
	
	Our next task is constructing the corresponding $\beta\gamma$-Heisenberg vertex algebra of $\sA_N$.
	Working directly with coordinates as above is subtle.
	We choose to start with $G$-invariants.
	
	Let $V_2$ be the $\beta\gamma$-Heisenberg vertex algebra corresponding to $\bA^2=\Spec\bC[\gamma^1,\gamma^2]$,
	and let $U(V_2)$ be the associated Lie algebra in \ref{va}.
	The $G$-invariants of the coordinate ring  $\bC[\gamma^1,\gamma^2]$ of $\bA^2$,
	with the action given by
	$$g\cdot\gamma^1= \epsilon \gamma^1,\quad
	g\cdot\gamma^2=\epsilon^{-1}\gamma^2\quad \text{for}\ n\leq0$$
	make up the coordinate ring of $\ul\sA_N$.
	Thanks to the natural embedding
	$$\bC[\gamma^1,\gamma^2]\to U(V_2)$$
	$$\gamma^i\mapsto\gamma^i_{[-1]},\quad i=1,2,$$
	and the coordinate change formula (\ref{trans}),
	we extend the $G$-action from $\bC[\gamma^1,\gamma^2]$ to the image of $U(V_2)\to \End V_2$ in the following way:
	\begin{equation}\label{hei}
		\begin{aligned}
			g&\cdot\gamma^1_m=\epsilon\gamma^1_m,\quad g\cdot\gamma^2_m=\epsilon^{-1}\gamma^2_m\quad\text{for}\ m\leq0,
			\\
			g&\cdot\beta^1_n=\epsilon^{-1}\beta^1_n,\quad g\cdot\beta^2_n=\epsilon\beta^2_n\quad\text{for}\ n<0.
		\end{aligned}
	\end{equation}
	Denote by $\overline{U(V_2)}$ the   image of $U(V_2)\to \End V_2$.
	Clearly $\overline{U(V_2)}\0=V_2$.
	We define the $\beta\gamma$-\textit{Heisenberg vertex algebra associated to}  $\sA_N$ to be the vertex subalgebra $\overline{U(V_2)}^G\0$ of $V_2$, denoted by $V_Q$. 
	
	\begin{rk}\label{G}
		The space $\overline{U(V_2)}^G\0$ of invariants  is generated by fields of the form
		$$\beta_{m_1}^i...\beta_{m_N}^i(z),\ \beta_m^1\beta_n^2(z),
		\ \gamma_{n_1}^i...\gamma_{n_N}^i(z),
		\ \gamma_m^1\gamma_n^2(z),
		\ 
		\beta_m^i\gamma_n^i(z)
		$$
		for $i=1,2$, $m<0$, $n\leq0$.
		It is straightforward to see that $Y(A,z)B\in\overline{U(V_2)}^G\0$ for 
		$A,B\in\overline{U(V_2)}^G\0$
		by checking the form of elements.
		Therefore $\overline{U(V_2)}^G\0$ is a vertex subalgebra as we claimed above.
	\end{rk}
	
	\begin{definition}
		The (global sections of) \textit{log chiral de Rham complex} on $\sA_N$ is defined as
		$\Omega^{\ch}_Q=V_Q\otimes \Lambda_Q.$
	\end{definition}
	
	\begin{rk}
		Since $L=\sum_{i=1,2}(\gamma_{-1}^i\beta_{-1}^i+c_{-1}^ib_{-1}^i)\in\Omega_Q^{\ch}$,
		the above $\Omega_Q^{\ch}$ is a conformal vertex algebra.
		However the element $Q=\beta_{-1}^1c_0^1+\beta_{-1}^2c_0^2$ does not belong to $\Omega_Q^{\ch}$,
		so unfortunately $\Omega_Q^{\ch}$ is not equipped with the chiral de Rham differential $Q_0$,
		which implies that it is no longer a complex.
		However, we are still able to study the algebraic structure of the chiral structure sheaf, which is the theme of the final section.
	\end{rk}

	\section{Character of log chiral structure sheaf}

	Recall that the chiral structure sheaf is associated to the $\beta\gamma$-Heisenberg vertex algebra $V_Q$ on $\sA_Q$.
	It is straightforward to see that the basis of $V_Q$ is as follows, by remark \ref{G}.
	
	\begin{prop}\label{vs}
		The vertex algebra $V_Q$ is isomorphic to
		$$\bC[\beta_{m_1}^i...\beta_{m_N}^i,\beta_m^1\beta_n^2,\gamma_{n_1}^i...\gamma_{n_N}^i,\gamma_m^1\gamma_n^2,\beta_m^i\gamma_n^i]_{i=1,2,\ m<0,\ n\leq0}$$
		as vector spaces.
	\end{prop}
	
	The Virasoro element $L_{V_2}=\beta^1_{-1}\gamma^1_{-1}+\beta^2_{-1}\gamma^2_{-1}$ of $V_2$
	belongs to $V_Q$, and thus
	$V_Q$ is a conformal vertex algebra graded by $(L_{V_2})_0$.
	Explicitly, the degree of a homogeneous vector in $V_Q$ is given by
	$$\deg\gamma^i_s=-s,\quad \deg\beta^i_r=-r$$
	for $i=1,2$, $s\in\bZ_{\leq0}$, $r\in\bZ_{<0}$.
	Let $V_Q^n$ be the subspace consisting of homogeneous elements in $V_Q$ of degree $n$,
	then it is easy to verify that $V_Q^m\cdot V_Q^n\subset V^{m+n}_Q$.
	We thus have a gradation 
	$$V_Q=\bigoplus_{n\in\bZ_{\geq0}} V_Q^n.$$
	In particular, we have
	$$V_Q^0=\bC[(\gamma_0^i)^N,\gamma_0^1\gamma_0^2]_{i=1,2}.$$
	Then $V_Q$ is a $V_Q^0$-algebra and each
	$V_Q^n$ is a $V_Q^0$-module for $n\geq0$.
	
	For positive integers $m$ and $n$,
	we define $p_n(m)$ to be the number of partitions of $m$ into at most $n$ terms,
	and set $p_n(k)=0$ for $k\leq0$.
	For example, $p_2(4)=3$ as $4=4+0=3+1=2+2$.
	\begin{thm}
		The lengths of homogeneous components of $V_Q$ are given by the following recursive formula:
		$$\operatorname{length}_{V^0_Q}V_Q^r=
		4r+2p_N(r)+2p_N(r-N)+
		\sum_{n_1+2n_2...+(r-1)n_{r-1}=r}
		\left(
		\begin{matrix}
			\operatorname{length}_{V^0_Q} V_Q^1\\
			n_1
		\end{matrix}
		\right)
		...
		\left(
		\begin{matrix}
			\operatorname{length}_{V^0_Q} V_Q^{r-1}\\
			n_{r-1}
		\end{matrix}
		\right)$$
		with initial condition $\operatorname{length}_{V^0_Q}V_Q^1=6$.
	\end{thm}
	
	\begin{proof}
		By Proposition \ref{vs} the degrees of the generators of $V_Q$ are
		\begin{equation*}
			\begin{aligned}
				\deg\beta_{m_1}^i...\beta_{m_N}^i&=-(m_1+...+m_N)\geq N\\
				\deg \beta_m^1\beta_n^2&=-m-n\geq2\\
				\deg\gamma_{n_1}^i...\gamma_{n_N}^i&=-({n_1+...+n_N})\geq0\\
				\deg\gamma_m^1\gamma_n^2&=-m-n\geq0\\
				\deg\beta_m^i\gamma_n^i&=-m-n\geq1.
			\end{aligned}
		\end{equation*}
		For $V^r_Q$, we separate it into two parts as follows. Let $v\in V^r_Q$:
		\\
		(1) $v$ arises from a product of elements of components of lower degree;
		\\
		(2) $v$ does not arise from (1).
		\\
		Part (1) contributes to the last summation of the formula.
		For part (2), when $r<N$,
		the generators and their numbers are
		\begin{equation*}
			\begin{aligned}
				\#\beta^1_m\beta^2_n&=\#\{(m,n):m+n=r,m,n\geq1\}=r-1,\\
				\# \gamma_{n_1}^i...\gamma_{n_N}^i&=\#\{(n_1,...,n_N):n_1+...+n_N=r,n_k\geq0\}=p_N(r)
				\\
				\#\gamma^1_m\gamma^2_n&=\#\{(m,n):m+n=r,m,n\geq0\}=r+1
				\\
				\#\beta^i_m\gamma_n^i&=\#\{(m,n):m+n=r,m\geq1,n\geq0\}=r
			\end{aligned}
		\end{equation*}
		So the length of the $V_Q^0$-submodule arising from part (2) is $4r+2p_N(r)$.
		In the case $r\geq N$,
		there will be extra generators $\beta_{m_1}^i...\beta_{m_N}^i$, and their number is
		$$\#\beta_{m_1}^i...\beta_{m_N}^i=\#\{(m_1,...,m_N):m_1+...+m_N=r, m_k\geq1\}=p_N(r-N),$$
		which contribute to a  $V_Q^0$-submodule of length $2p_N(r-N)$.
		And thus for $r\geq N$, the length of the $V_Q^0$-submodule arsing from part (2) is $4r+2p_N(r)+2p_N(r-N)$.
		In summary, since $p_N(r-N)=0$ for $r<N$,
		we obtain the length of $V_Q^r$ over $V_Q^0$ as claimed.
		The length of $V^1_Q$ can be easily obtained by 
		$$V_Q^1=\Span_{V^0_Q}\{
		(\gamma_0^i)^{N-1}\gamma^i_{-1},\gamma_0^1\gamma_{-1}^2,\gamma_{-1}^1\gamma_0^2,\beta_{-1}^i\gamma_0^i
		\}_{i=1,2}.$$
		This completes the proof.
	\end{proof}

\end{document}